\documentclass{amsart}
\usepackage{}
\usepackage{color}
\font\bbbld=msbm10 scaled\magstephalf

\newcommand{\bM}{\bar{M}}

\newcommand{\bfR}{\hbox{\bbbld R}}

\newcommand{\tF}{\tilde{F}}

\newcommand{\ul}{\underline}

\newtheorem{theorem}{Theorem}[section]
\newtheorem{lemma}[theorem]{Lemma}
\newtheorem{proposition}[theorem]{Proposition}

 \theoremstyle{definition}

\theoremstyle{remark}
\newtheorem{remark}[theorem]{Remark}

\numberwithin{equation}{section}

\begin{document}

\title[fully nonlinear parabolic equations]
{Estimates for a class of hessian type fully nonlinear parabolic equations
on Riemannian manifolds}
\author{Ge-jun Bao}
\address{Department of Mathematics, Harbin Institute of Technology,
         Harbin, 150001, China}
\email{baogj@hit.edu.cn}
\author{Wei-song Dong*}
\address{Department of Mathematics, Harbin Institute of Technology,
         Harbin, 150001, China}
\email{dweeson@gmail.com}


\begin{abstract}

In this paper, we derive \emph{a priori} estimates for the gradient and second order derivatives of
solutions to a class of Hessian type fully nonlinear parabolic equations with the first initial-boundary
value problem on Riemannian manifolds.
These \emph{a priori} estimates are derived under conditions which are nearly optimal. Especially, there
are no geometric restrictions on the boundary of the Riemannian manifolds.
And as an application, the existence of smooth solutions to the first initial-boundary value problem even for infinity time
is obtained.



\end{abstract}

\maketitle

\section{Introduction}

Let $(M^n, g)$ be a compact Riemannian manifold of
dimension $n \geq 2$ with smooth boundary $\partial M$
and $\bM := M \cup \partial M$.
We study the Hessian type fully nonlinear parabolic equation
\begin{equation}
\label{eqn}
f (\lambda[\nabla^{2}u + \chi], - u_t)  = \psi (x, t)
\end{equation}
in $M_T = M \times (0,T] \subset M \times \mathbb{R}$, where $f$ is
a symmetric smooth function of $n + 1$ variables, $\nabla^2 u$ denotes the
Hessian of $u (x, t)$ with respect to the space $x \in M$, $u_t$ is the derivative
with respect to the time $t \in (0, T]$, $\chi$ is a smooth
(0, 2) tensor on $\bM$
and $\lambda [\nabla^{2} u + \chi] = (\lambda_1 ,\ldots,\lambda_n)$ denotes the
eigenvalues of $\nabla^{2} u + \chi$ with respect to the metric $g$. While the first initial-boundary
value problem requires:
\begin{equation}
\label{eqn-b}
 u = \varphi, \; \mbox{on $\mathcal{P}M_T$},
\end{equation}
where $\mathcal{P} M_T = B M_T \cup S M_T$ and $B M_T = M \times \{0\}$,
$S M_T = \partial M \times [0,T]$.
We assume $\psi \in C^{4, 1} (\overline {M_T}) $, $\varphi \in C^{4, 1}(\mathcal{P}M_T)$.

As in \cite{CNS}, we assume $f \in C^\infty (\Gamma) \cap C^0 (\overline{\Gamma})$
to be defined on an open, convex, symmetric proper subcone $\Gamma \subset \mathbb{R}^{n + 1}$ with vertex at the origin and
\[\Gamma^{+} \equiv \{\lambda \in \mathbb{R}^{n + 1}:\mbox{ each component }
   \lambda_{\ell} > 0, \; 1 \leq \ell \leq n+1\} \subseteq \Gamma.\]
In this work we assume only a few conditions on $f$, which are almost optimal, but the followings are essential as the structure conditions.
We assume that $f$ satisfies:
\begin{equation}
\label{f1}
f_{\ell} \equiv \frac{\partial f}{\partial \lambda_{\ell}} > 0 \mbox{ in } \Gamma,\ \ 1\leq \ell \leq n + 1,
\end{equation}
\begin{equation}
\label{f2}
f\mbox{ is concave in }\Gamma,
\end{equation}
and
\begin{equation}
\label{f5}
f > 0 \mbox{ in } \Gamma, \ \ f = 0 \mbox{ on } \partial \Gamma, \ \ \inf_{M_T}\psi > 0.
\end{equation}

In \cite{G} Guan has developed some methods to derive \emph{a priori} second order estimates under nearly
optimal conditions for solutions of a class of fully nonlinear elliptic equations on Riemannian manifolds.
More recently, Guan and Jiao \cite{GJ} further
developed the methods to cover more general elliptic equations. In this paper we prove the mechanism in \cite{G}
to derive second order estimates is also valid for a wild class of Hessian type fully nonlinear parabolic equations on Reimannian
manifolds.
Following \cite{GJ} we assume:
\begin{equation}
\label{f3}
T_{\lambda} \cap \partial \Gamma^{a} \; \mbox{is a
nonempty compact set, for any $\lambda \in \Gamma$ and $0 < a < f (\lambda)$},
\end{equation}
where
$\partial \Gamma^{\sigma} = \{\lambda \in \Gamma: f(\lambda) = \sigma\}$ is
the boundary of $\Gamma^{\sigma} = \{\lambda \in \Gamma: f (\lambda) > \sigma\}$
and $T_{\lambda}$ denotes
the tangent plane at $\lambda$ of $\partial \Gamma^{f (\lambda)}$,
for $\sigma \in \bfR^+$ and $\lambda \in \Gamma$. By
assumptions~(\ref{f1}) and (\ref{f2}), $\partial \Gamma^{\sigma}$  is smooth and convex.

Since we need no geometric boundary conditions, we have to assume,
and which is more convenience in application, that there exists an admissible function (see Section 2)
$\underline{u} \in C^{2, 1} (\overline {M_T})$ satisfying
\begin{equation}
\label{sub}
f(\lambda[\nabla^{2} \underline{u} + \chi], - \underline{u}_{t})  \geq \psi(x,t) \mbox{ in } M_T
\end{equation}
with $\ul u = \varphi$ on $\partial M \times (0, T]$ and $\ul u \leq \varphi$ on $\bM \times \{0\}$, which we call a \emph{subsolution}.
If the inequality \eqref{sub} holds strictly, then we call $\ul u$ a \emph{strict} subsolution.
In \cite{L}, Lieberman proved that there exists a strict subsolution
under conditions that for any compact
subset $K$ of $\overline {M_T} \times \Gamma$,
there exists a positive constant $R(K)$ such that
$ f(R\lambda) > \psi (x, t)$
for any $R \geq R(K)$, $(x, t, \lambda) \in K$,
and that there is a positive constant $R_1$ such that
$(\kappa, R_1) \in \Gamma$,
where $\kappa = (\kappa_0, \ldots, \kappa_{n - 1})$ is the space-time curvatures of $SM_T$(see \cite{L}).

Without loss of generality, we assume the compatibility condition, that is for all $x \in \bM$,
$(\lambda[\nabla^{2} \varphi(x,0) + \chi], -\varphi_t(x, 0)) \in \Gamma$, and
\begin{equation}
\label{comp}
f (\lambda [\nabla^2 \varphi (x, 0) + \chi(x)], - \varphi_t (x, 0))  = \psi (x, 0).
\end{equation}
We remark that this condition is actually ensured by the short time existence of solution to equation \eqref{eqn} and \eqref{eqn-b}.
Now we can give out our main result as below.

\begin{theorem}
\label{bd-th1}
Let $u \in C^{4, 1} (M_T) \cap C^{2, 1}(\overline {M_T})$ be an
admissible solution of (\ref{eqn}) in $M_T$ with $u = \varphi \; \mbox{on} \; \mathcal{P} M_T$.
Suppose $f$ satisfies \eqref{f1} - \eqref{sub}, the following and \eqref{comp} hold:
\begin{equation}
\label{c-290}
\sum_{\ell =1}^{n+1} f_\ell \lambda_\ell \geq - K_0 (1 + \sum_{\ell = 1}^{n + 1} f_\ell), \; \; \forall \lambda \in \Gamma.
\end{equation}
Then we have
\begin{equation}
\label{gsui-1}
\max_{\overline {M_T}} (|\nabla^2 u| + |u_t| ) \leq C,
\end{equation}
where $C > 0$ depends on $|u|_{C^1_x (\overline {M_T})}$,
$|\underline{u}|_{C^{2,1} (\overline {M_T})}$, $|\psi|_{C^{2, 1}(\overline {M_T}))}$,
$|\varphi|_{C^1_t(\overline {M_T}))}$ and other known data.
\end{theorem}

We remark that in Theorem \ref{bd-th1}, \eqref{c-290} is only needed when deriving second order boundary estimates,
and there the norms are defined as below:
\[\begin{aligned}|u|_{C^1_x (\overline {M_T})} = & |u|_{C^0 (\overline {M_T})} + |\nabla u|_{C^0 (\overline {M_T})},\;
|\varphi|_{C^1_t (\overline {M_T})} = |\varphi_t|_{C^0 (\overline {M_T})},\\
|\psi|_{C^{2, 1} (\overline {M_T})} = & |\psi|_{C^0 (\overline {M_T})}
+ |\nabla \psi|_{C^0 (\overline {M_T})} + |\nabla^2 \psi|_{C^0 (\overline {M_T})} + |\psi_t|_{C^0 (\overline {M_T})}.
\end{aligned}\]

For the gradient estimates, firstly from
$\Gamma \subset \{ \lambda \in \mathbb{R}^{n + 1}: \sum_{\ell =1}^{n + 1} \lambda_\ell \geq 0\}$,
we see that $u$ is a subsolution of
\begin{equation}
\label{ups}
\left\{\begin{aligned} - h_t +  \triangle h + \;& \sum \chi_{ii} (x, t) =  0, \; & \mbox{in} \; M_T,\\
h = \;& \varphi, \; & \mbox{on} \; \mathcal{P} M_T.
\end{aligned}\right.
\end{equation}
If we assume $h$ is the solution of the above linear equation,
we may easily get $u \leq h$ on $\overline {M_T}$ by the comparison principle. On the other hand,
since $\ul u$ is a subsolution of \eqref{eqn} and \eqref{eqn-b}, we have $\ul u \leq u$.
Therefore, we have the $C^0$ estimates that
\begin{equation}
\label{C0}
\sup_{M_T} |u| + \sup_{S M_T} |\nabla u| \leq C.
\end{equation}
While it is evident that on $B M_T$, we have $\nabla u = \nabla \varphi$. So the following theorem
completes the main work of this paper.
\begin{theorem}
\label{bd-th2}
Suppose $f$ satisfies \eqref{f1} - \eqref{f2}.
Let $u \in C^{4, 1} (M_T) \cap C^{2, 1}(\overline {M_T})$ be an
admissible solution of (\ref{eqn}) in $M_T$. Then
\begin{equation}
\label{gsui-2}
\sup_{M_T}|\nabla u| \leq C ( 1 + \sup_{\mathcal{P} M_T}|\nabla u|),
\end{equation}
where $C$ depends on $|\psi|_{C^1_x (\overline {M_T})}$,
$|u|_{C^0 (\overline {M_T})}$ and other known data, under either of the following additional assumptions:
$(\mathbf{i})$ $f$ satisfies \eqref{c-290} and
\begin{equation}
\label{f6}
f_\jmath (\lambda) \geq \nu_1 ( 1 + \sum_{\ell =1}^{n+1} f_\ell (\lambda) ) \mbox{ for any }
  \lambda \in \Gamma \mbox{ with } \lambda_\jmath < 0,
\end{equation}
where $1 \leq \jmath \leq n + 1$, and $\nu_1$ is a uniform positive constant;
$(\mathbf{ii})$ \eqref{f5} and \eqref{sub} hold, as well as that $(M, g)$ has nonnegative sectional curvature.
\end{theorem}

Based on the above \emph{a priori} estimates and \eqref{f5}, equation \eqref{eqn}
becomes uniformly parabolic equation. Then
by Evans-Krylov Theorem \cite{Evans,K}, we can obtain the \emph{a priori} $C^{2 + \alpha, 1 + \alpha/2}$
estimates. Therefore it is possible to apply the theory of linear uniformly
parabolic equations (see \cite{L} for more) to get higher order estimates.
We remark that, the \emph{a priori} estimates in Theorem \ref{bd-th1}, Theorem \ref{bd-th2} and \eqref{C0} do not depend on
the time $t$ explicitly, and as a byproduct of these estimates, we have the following (long time, i.e. $T = \infty$) existence results.
We note that a function in $C^{\infty} (\overline {M_T})$ means that it is sufficiently smooth about $(x, t) \in \overline {M_T}$,
and note $M_\infty = M \times \{t > 0\}$.
\begin{theorem}
\label{jsui-th4}
Let $\psi \in C^{\infty} (\overline {M_T}) $ and $\varphi \in C^{\infty} (\mathcal{P} M_T)$, $0 < T \leq \infty$.
Suppose $f$ satisfies \eqref{f1} - \eqref{sub}, and \eqref{comp} holds. In addition that
either \eqref{c-290} and \eqref{f6} or
$(M^n, g)$ has nonnegative sectional curvature holds. Then
there exists a unique admissible solution $u \in C^{\infty} (\overline {M_T})$ to
the first initial-boundary value problem
\eqref{eqn} and \eqref{eqn-b}.
\end{theorem}

The above theorem is a direct result of the short time existence and the uniform estimates
in Theorem \ref{bd-th1} and Theorem \ref{bd-th2}, because at each
beginning time we can take $\varphi = u$,
which enables us to assume the compatibility condition, for more one can see Theorem 15.9 in \cite{L}.

Here we give some typical examples of our equation,
for example $f = \sigma^{1/k}_k$ for $k \geq 2$ (the reason why $k$ cannot equal to $1$
is due to condition \eqref{f3}, see \cite{G}) and $f = (\sigma_k / \sigma_l)^{1/(k - l)}$,
$1 \leq l < k \leq n + 1$, both of which are defined on the cone
$\Gamma_{k} = \{\lambda \in \mathbb{R}^{n+1}: \sigma_{j} (\lambda) > 0, j = 1, \ldots, k\}$,
where $\sigma_{k} (\lambda)$ are the elementary symmetric functions
$\sigma_{k} (\lambda) = \sum_ {i_{1} < \ldots < i_{k}}
\lambda_{i_{1}} \ldots \lambda_{i_{k}}$.
Another interesting example is given by $f = \log P_k$, where
\[P_k (\lambda) = \prod_{i_1 < \cdots < i_k} (\lambda_{i_1} + \cdots + \lambda_{i_k}), \ \ 1 \leq k \leq n + 1,\]
defined in the cone
$\mathcal{P}_k := \{\lambda \in \mathbb{R}^{n + 1}: \lambda_{i_1} + \cdots + \lambda_{i_k} > 0\}$.

Krylov in \cite{K1} introduced three parabolic type equations analogous to  Monge-Amp\`{e}re equation in $\mathbb{R}^n$.
One type which is studied on Riemannian manifolds by Jiao and Sui in \cite{JS} recently is
\begin{equation}
\label{eqn-js}
f (\lambda[\nabla^{2}u + \chi]) - u_t = \psi (x, t),
\end{equation}
under assumptions that
$\inf_{\mathcal{P} M_T} (\varphi_t + \psi) = \nu_0 > 0$,
and
$\psi (x, t)$ is concave with respect to $x \in M$.
This type equation in $\mathbb{R}^n$ when $f = \sigma_n^{1/n}$
with $\chi \equiv 0$
was firstly considered by Ivochkina and Ladyzhenskaya in \cite{IL1} and \cite{IL2}.
Another type is
\begin{equation}
\label{MA-p}
- u_t \det(\nabla^2 u) = \psi^{n + 1},
\end{equation}
which is a typical form of our case in $\mathbb{R}^{n + 1}$ with $\chi \equiv 0$.
Some other cases can be fined in Chou and Wang \cite{CW2001} or Wang \cite{WX2}.

At the end of the introduction, we describe the outline of our paper.
In Section 2, we state some preliminaries and introduce our main tool (Theorem \ref{3I-th3})
to establish the $C^2$ \emph{a priori} estimates, and two propositions which are needed when
deriving the second order estimates on boundary.
In Section 3, we establish the estimates for $|u_t|$ that do not depend on $T$ explicitly,
after which we have a bound for the
constant $C(\epsilon, |u_t|, K_0, \sup_{M_T} \psi)$ in Proposition \ref{par-2}.
Then the mechanism in \cite{G} is valid for the second order boundary estimates,
and the global and boundary estimates for second order derivatives are derived in
Section 4 and Section 5 respectively.
In Section 6, we establish the interior gradient estimates as the end.

\section{Preliminaries}
\label{gj-P}
\setcounter{equation}{0}
\medskip
From now on, we stipulate that the Latin alphabet $i, j, k, \cdots $ are valued
between $1$ and $n$ when there is no other statement.
Firstly, we give some notations and formulas on Riemannian manifolds,
throughout the paper $\nabla$ denotes the Levi-Civita connection
of $(M^n, g)$.
Let $e_1, \ldots, e_n$ be a local frame on $M^n$. We denote
$g_{ij} = g (e_i, e_j)$, $\{g^{ij}\} = \{g_{ij}\}^{-1}$.
Define the Christoffel symbols $\Gamma_{ij}^k$ by
$\nabla_{e_i} e_j = \Gamma_{ij}^k e_k$ and the curvature
coefficients
\[  R_{ijkl} = g( R (e_k, e_l) e_j, e_i), \;\; R^i_{jkl} = g^{im} R_{mjkl}.  \]
We shall use the notation $\nabla_i = \nabla_{e_i}$,
$\nabla_{ij} = \nabla_i \nabla_j - \Gamma_{ij}^k \nabla_k $, etc. Finally we recall the following formula
on Riemannian manifolds
\begin{equation}
\label{hess-A70}
\nabla_{ijk} v - \nabla_{jik} v = R^l_{kij} \nabla_l v,
\end{equation}
which will be frequently used in following sections.

Let $u$ be an admissible solution of equation~\eqref{eqn}.
For simplicity we define $U \equiv \nabla^2 u + \chi$,
$\ul U \equiv \nabla^2 \ul u + \chi$
and under an orthonormal local frame $e_1, \ldots, e_{n}$, we write
$U_{ij} \equiv U (e_i, e_j) = \nabla_{ij} u + \chi_{ij}$.
Direct calculating yields that
\[\begin{aligned}
\nabla_k U_{ij} \equiv & \nabla U (e_k, e_i, e_j)
                           =   \nabla_k \chi_{ij} + \nabla_{kij} u, \\
\nabla_{kl} U_{ij} \equiv & \nabla^2 U (e_k, e_l, e_i, e_j )
                           =  \nabla_k \nabla_l U_{ij} - \Gamma_{kl}^m \nabla_m U_{ij}.
\end{aligned}\]
For the convenience, sometimes we denote $- u_t$ by $U_{n+1 n+1}$, i.e. $U_{n+1 n+1} = -u_t$,
and $U_{i n+1} = U_{n+1 i} = 0$ where $1 \leq i \leq n$.

Let $F$ be the function defined by $F (A) = f (\lambda [A])$
for $A \in \mathbb{S}^{n + 1}$ with $\lambda [A] \in \Gamma$,
where $\mathbb{S}^{n + 1}$ is the set of $(n + 1) \times (n + 1)$ symmetric matrices.
We call a function $u \in C^{2, 1}(M_T)$ \emph{admissible} if
$(\lambda[\nabla^{2} u + \chi], - u_t) \in \Gamma$ in $M_T$.
It is shown in \cite{CNS} that \eqref{f1} ensures that equation \eqref{eqn}
is parabolic (i.e. $\{\partial F (A) / \partial A_{ij}\}$ is positive
definite) with respect to admissible solutions, while
\eqref{f2} implies that the function $F$ is concave.
For an admissible solution  $u \in C^{2, 1}(M_T)$  denote
\[
\hat{U} \equiv [U_{ij}, - u_t] \equiv
\left(                  
  \begin{array}{cc}     
    U_{ij} & 0\\   
    0 & -u_t\\   
  \end{array}
\right),
\]
under an orthonormal local frame $e_1, \ldots, e_n$. Therefore equation \eqref{eqn} can be locally written as
\begin{equation}
\label{eqn'}
 F (\hat{U}) = \psi (x, t).
\end{equation}
We denote
\[ F^{ij} = \frac{\partial F}{\partial A_{ij}} (\hat{U}), \;\;
  F^{ij, kl} = \frac{\partial^2 F}{\partial A_{ij} \partial A_{kl}} (\hat{U}), \;\;
  F^{\tau} = \frac{\partial F}{\partial A_{n + 1, n + 1}} (\hat{U}) \equiv f_\tau. \]
The matrix $[\{F^{ij}\}, F^\tau]$ has eigenvalues $f_1, \ldots, f_n, f_\tau$ and
is positive definite by assumption \eqref{f1}.
Moreover, when $[\{U_{ij}\}, - u_t]$ is diagonal so is $[\{F^{ij}\}, F^\tau]$, and as in \cite{G} the
following identities hold
\[   F^{ij} U_{ij} = \sum f_i \lambda_i, \;\; F^{ij} U_{ik} U_{kj} = \sum f_i \lambda_i^2,\]
where $\lambda (\{U_{ij}\}) = (\lambda_1, \ldots, \lambda_{n})$.

The following theorem proved in \cite{GJ} is the keystone in deriving \emph{a priori} $C^2$ estimates in our paper.
\begin{theorem}
\label{3I-th3}
Suppose $f$ satisfies \eqref{f1}, \eqref{f2} and \eqref{f3}.
Let $\Im $ be a compact set of $\Gamma$ and
$\sup_{\partial \Gamma} f < a \leq b < \sup_{\Gamma} f$.
There exist positive constants $\theta = \theta (\Im, [a, b])$
and $R = R (\Im, [a, b])$ such that for any
$\lambda \in \Gamma^{[a, b]} = \bar \Gamma^a \backslash \Gamma^b$,
when $|\lambda| \geq R$,
\begin{equation}
\label{3I-100}
\sum_{\ell=1}^{n + 1} f_\ell (\lambda) (\mu_{\ell} - \lambda_\ell) \geq
     \theta + \theta \sum_{\ell=1}^{n + 1} f_\ell (\lambda) + f (\mu) - f (\lambda), \; \forall \mu \in \Im.
\end{equation}
\end{theorem}
For $\forall \, v \in C^{2, 1} (\overline {M_T})$, we define the linear operator $\mathcal{L}$ by
\[\mathcal{L} v = F^{ij} \nabla_{ij} v - F^\tau v_t .\]
Choose a smooth orthonormal local frame
$e_1, \ldots, e_n$ about $(x, t)$ such that $\{U_{ij}(x,t)\}$ is diagonal.
From Lemma 6.2 in \cite{CNS} and Theorem \ref{3I-th3}, it is easily
to prove that there exist positive constants $\theta$, $R$ depending only
on $\ul u$ and $\psi$ such that when $|\lambda| = |\lambda [\hat{U}]| \geq R$,
\begin{equation}
\label{gj}
\mathcal{L} (\ul u - u) = F^{ii} ( \ul{U_{ii}} - U_{ii} ) + F^{\tau} (-\ul u_t + u_t) \geq \theta (1+ \sum F^{ii} + F^{\tau}).
\end{equation}

\begin{remark}
\label{rmk-0}
If $\ul u$ is a strict subsolution.
Note that $\{\lambda (\ul {\hat{U}}) : (x, t) \in \overline {M_T}\}$
is contained in a compact
subset of $\Gamma$, here
$\ul {\hat{U}} =  [\ul U, - \ul u_t]$.
We see that there exist constants $\varepsilon, \delta_0 > 0$ such that $\lambda
[\ul U (x, t) - \varepsilon g, -\ul {u}_t - \varepsilon] \in \Gamma$ for all $(x, t) \in \overline {M_T}$ and
$F ( [\ul U - \varepsilon g, - \ul u_t - \varepsilon])  \geq \psi + \delta_0$.
By the concavity of $F$, we have
\[
\begin{aligned}
 F^{ij} ( \ul{U}_{ij} - U_{ij} ) - F^{\tau} ( \ul{u}_t - u_t )
   \geq \varepsilon (\sum F^{ii} + F^\tau) + \delta_0.
\end{aligned}
\]
That means \eqref{gj} is valid in the whole $\overline {M_T}$ if $\ul u$ is a strict subsolution.
\end{remark}

The following two propositions play the key role in the second order boundary estimates,
which are the generalized counterpart results in \cite{G}.
\begin{proposition}
\label{par-1}
Let $F(\hat U) = f (\lambda (U), - u_t)$. There is an index $1 \leq r \leq n$ such that
\begin{equation}
\label{par-11}
\sum_{l<n} F^{ij} U_{il} U_{lj} \geq \frac{1}{2} \sum_{i \neq r} f_i \lambda_i^2.
\end{equation}
\end{proposition}
This can be proved by exactly the same method as the prove of
Proposition 2.7 in \cite{G}. So we omit the proof.

One more result we need is the following which actually is a combination of
generalized Lemma 2.8 and Corollary 2.9 in \cite{G}. The method of this proof
is from \cite{GSS}.
\begin{proposition}
\label{par-2}
Suppose $f = f (\lambda (U), - u_t)$ satisfies \eqref{f1}, \eqref{f2} and \eqref{c-290}. Then for any index $1 \leq r \leq n$ and
$\epsilon > 0$,
\begin{equation}
\label{par-21}
\sum f_i |\lambda_i| \leq \epsilon \sum_{i \neq r} f_i \lambda_i^2 +
C(\epsilon, |u_t|, K_0, \sup_{M_T} \psi) (1 + \sum f_i + f_\tau).
\end{equation}
\end{proposition}

\begin{proof}
Firstly, if $\lambda_r \leq 0$, by \eqref{c-290}, we have
\[\begin{aligned}
\sum f_i |\lambda_i|
= & 2 \sum_{\lambda_i > 0} f_i \lambda_i - \sum f_i \lambda_i\\
\leq & \epsilon \sum_{\lambda_i > 0} f_i \lambda_i^2 +
\frac{1}{\epsilon} \sum_{\lambda_i > 0} f_i + f_\tau (- u_t) + K_0 (1 + \sum f_i + f_\tau)\\
\leq & \epsilon \sum_{i \neq r} f_i \lambda_i^2 +
C (\epsilon, K_0) \sum f_i  + \max \{|u_t|, K_0\} f_\tau + K_0.
\end{aligned}
\]

Secondly, if $\lambda_r \geq 0$, then by \eqref{f2}, we have
\[
\begin{aligned}
\sum f_i |\lambda_i| = & \sum f_i \lambda_i - 2 \sum_{\lambda_i < 0} f_i \lambda_i\\
                     \leq & \epsilon \sum_{\lambda_i < 0} f_i \lambda_i^2 + \frac{1}{\epsilon}
                      \sum_{\lambda_i < 0} f_i + \sum f_i + f_\tau + f_\tau u_t
                        + \psi - f (\mathbf{1})\\
                        \leq & \epsilon \sum_{\lambda_i < 0} f_i \lambda_i^2
                         + C (\epsilon, |u_t|, \sup_{M_T} \psi) (1 + \sum f_i + f_\tau)
 \end{aligned}
\]
since $f (\mathbf{1}) > 0$ and where $\mathbf{1} = (1, \cdots, 1) \in \mathbb{R}^{n + 1}$.
This proves \eqref{par-21}.
\end{proof}

\section{estimates for $u_t$ }
The assumption \eqref{f3} is crucial for the estimates of $u_t$. In a forthcoming paper we will
consider the estimates of $u_t$ without this restriction.
By the compatibility condition \eqref{comp}, on $BM_T$, we have $u_t = \varphi_t$, which is also valid apparently on
$ \partial M \times [0, T]$. Hence we have
$\sup_{\mathcal{P}M_T} |u_t| \leq C$.
Now by differentiating equation \eqref{eqn'} with respect to $t$ we see that
$\mathcal{L} u_t  =  \psi_t$.
Let $a$ be a positive constant to be determined, by \eqref{gj}, if $|u_t|$ is sufficiently large, by Theorem \ref{3I-th3}, we have
\begin{equation}
\mathcal{L}(- u_t + a (\ul u - u))  \geq - C + a \theta (1 + \sum F^{ii} + F^{\tau})
                                    \geq 0
\end{equation}
when $a \geq \frac{C}{\theta}$.
Similarly we can prove the same result holds for $u_t$.
Thus by maximum principle we have
\begin{equation}
\sup_{M_T} |u_t| \leq C .
\end{equation}
here $C$ depends on $|\psi|_{C^1_t}$, $|u|_{C^0}$, $|\varphi|_{C^1_t}$and other known data.

Therefore, the estimate for $|u_t|$ implies that the constants $C(\epsilon, |u_t|, K_0, \sup_{M_T} \psi)$
in Proposition \ref{par-2} are bounded, which enables us to apply the
mechanism in \cite{G} to derive the second order boundary estimates.

\section{$C^{2}$ global estimates}

In this section, we derive \emph{a priori} global estimates for the second order derivatives. We set
\[W = \max_{(x,t) \in \overline {M_T}} \max_{\xi \in T_x M, |\xi| = 1}
     (\nabla_{\xi\xi} u + \chi_{\xi \xi})
        \exp (\frac{\delta}{2} |\nabla u|^{2} + a (\underline{u} - u)),\]
where $a \gg 1 \gg \delta$
are positive constants to be determined later. It suffices to estimate $W$.
We may assume $W$ is achieved at $(x_{0}, t_{0}) \in \overline {M_T} - \mathcal{P} M_T$
for some unit vector $\xi \in T_{x_0} M^n$.
Choose a smooth orthonormal local frame $e_{1}, \ldots, e_{n}$ about $x_{0}$
such that $e_1 (x_0) = \xi$, $\nabla_{e_i} e_j = 0$,  and
$\{U_{ij} (x_0, t_0)\}$ is diagonal.
We may also assume $U_{11} \geq \ldots \geq U_{n n}$, $U_{11} \geq \sup_{M_T} |u_t|$.
Therefore $W = U_{11} (x_0, t_0) e^{\phi (x_0, t_0)}$,
where $\phi = \frac{\delta}{2} |\nabla u|^{2} + a (\underline{u} - u)$,
and $|U_{ii}| \leq n |U_{11}|$ which derived from $- u_t + U_{11} + \cdots + U_{nn} > 0$.

At the point $(x_{0}, t_{0})$ where the function
$\log (U_{11}) + \phi$ attains its maximum,
we have
\begin{equation}
\label{gs3}
\frac{\nabla_i U_{11}}{U_{11}} + \nabla_i \phi = 0
    \mbox{ for each } i = 1, \ldots, n,
\end{equation}
\begin{equation}
\label{gs4}
\frac{(\nabla_{11} u)_t}{U_{11}} + \phi_t \geq 0,
\end{equation}
and
\begin{equation}
\label{gs5}
0 \geq \sum_{i = 1}^{n} F^{ii} \Big\{ \frac{\nabla_{ii} U_{11}}{U_{11}}
   - \frac{(\nabla_i U_{11})^{2}}{U^{2}_{11}} + \nabla_{ii} \phi \Big\}.
\end{equation}
Differentiating equation (\ref{eqn}) twice, we obtain
\begin{equation}
\label{gs1}
F^{ij} \nabla_{k} U_{ij} - F^{\tau} \nabla_k u_t = \nabla_k \psi \;\; \mbox{ for } \; 1 \leq k \leq n,
\end{equation}
and
\begin{equation}
\begin{aligned}
\label{gs2}
F^{ij} \nabla_{11} U_{ij} + & \, F^{ij,kl} \nabla_1 U_{ij} \nabla_1 U_{kl}
  + F^{\tau\tau} (\nabla_1 u_t)^2 \\
 -& \, 2 F^{ij, \tau} \nabla_1 U_{ij} \nabla_1 u_t - F^\tau \nabla_{11} u_t
 =  \nabla_{11} \psi .
\end{aligned}
\end{equation}

Hence, combining \eqref{gs3}, \eqref{gs4}, \eqref{gs5} and \eqref{gs2}, and noting $ \nabla_{ii}U_{11} \geq \nabla_{11} U_{ii} - C U_{11}$
(see \cite{G}), we have
\begin{equation}
\label{gs11}
 \mathcal{L} \phi  \leq  E + C\Big(1 + \sum F^{ii} \Big)
\end{equation}
when $U_{11}$ is sufficiently large,
where
\[
\begin{aligned}
 E = \, &  \frac{F^{ii} (\nabla_{i} U_{11})^{2}}{{U^{2}_{11}}}
 + \frac{  F^{ij,kl} \nabla_1 U_{ij} \nabla_1 U_{kl}
  - 2 F^{ij, \tau} \nabla_1 U_{ij} \nabla_1 u_t
   + F^{\tau\tau} (\nabla_1 u_t)^2 }{U_{11}}.
\end{aligned}
\]
By some straightforward calculation, we have, at $(x_0, t_0)$,
\begin{equation}
\label{test1}
\nabla_i \phi =  \delta \nabla_j u \nabla _{ij} u + a \nabla_i (\ul u - u), \;\;
        \phi_t =  \delta \nabla_j u (\nabla_j u)_t + a (\ul u - u)_t,
\end{equation}
\begin{equation}
\label{test2}
\begin{aligned}
\nabla_{ii} \phi = \,& \delta (\nabla_{ij} u \nabla_{ij} u
    + \nabla_j u \nabla_{iij} u) + a \nabla_{ii} (\ul u - u)\\
      \geq \,& \frac{\delta}{2} U_{ii}^2 - C \delta
         + \delta \nabla_{j} u \nabla_{iij} u + a \nabla_{ii} (\ul u - u).
\end{aligned}
\end{equation}
Thus, by \eqref{gs1} and \eqref{hess-A70} we have,
\begin{equation}
\label{test4}
\begin{aligned}
F^{ii} \nabla_{ii} \phi \geq \,& \frac{\delta}{2} F^{ii} U^2_{ii}
+ \delta \nabla_j u F^{ii} (\nabla_{jii} u + R^l_{iij} \nabla_l u)\\
    & + a F^{ii} \nabla_{ii} (\ul u - u)  - C \delta \sum F^{ii}\\
     \geq \,& \frac{\delta}{2} F^{ii} U^2_{ii} + \delta F^{\tau} \nabla_{j} u \nabla_j u_t + a F^{ii} \nabla_{ii} (\ul u - u) \\
     \, &  - C \delta (1 + \sum F^{ii}).
\end{aligned}
\end{equation}
Therefore, by \eqref{gs11}, \eqref{test1} and \eqref{test4} we obtain
\begin{equation}
\label{gs12}
\begin{aligned}
 a \mathcal{L} (\ul{u} - u) \,& \leq  E  - \frac{\delta}{2} F^{ii} U^2_{ii}
      +  C ( 1+ \sum F^{ii} ).
\end{aligned}
\end{equation}

For fixed $0 < s \leq 1/3$, let
\[
J  = \{i: U_{ii} \leq - s U_{11}, 1 < i \leq n\}, \;\;
K  = \{i:  U_{ii} > - s U_{11}, 1 \leq i \leq n\}.
\]
Using a result of Andrews \cite{A} and Gerhardt \cite{GC}
(see \cite{U} also), and noting that $U_{n+1 j} = 0$ for all $j = 1, 2, \cdots, n$, we have,
\begin{equation}
\label{gj-S130}
\begin{aligned}
 - F^{ij, kl} \nabla_1 U_{ij} \nabla_1 U_{kl} + \,& 2 F^{ij, \tau} \nabla_1 U_{ij} \nabla_1 u_t - F^{\tau\tau} \nabla_1 u_t \nabla_1 u_t\\
\geq \,& \sum_{1 \leq i \neq j \leq n + 1} \frac{F^{ii} - F^{jj}}{U_{jj} - U_{ii}}
           (\nabla_1 U_{ij})^2 \\
 \geq \,& 2 \sum_{2 \leq i \leq n} \frac{F^{ii} - F^{11}}{U_{11} - U_{ii}}
            (\nabla_1 U_{i 1})^2 \\
\geq \,& \frac{2 (1-s)}{(1+s) U_{11}} \sum_{i \in K} (F^{ii} - F^{11})
            ((\nabla_i U_{11})^2 - C /s),
\end{aligned}
\end{equation}
where in the last inequality we used the following result which can be readily proved with \eqref{hess-A70},
that for any $ s \in (0, 1)$
\begin{equation}
 \label{chu1}
 \begin{aligned}
  (1 - s)(\nabla_i U_{11})^2 \leq (\nabla_1 U_{1i})^2 + C (1 - s)/s .
 \end{aligned}
\end{equation}
From \eqref{gj-S130}, combining \eqref{gs3} and that $\nabla_i \phi \leq \delta \nabla_i u U_{ii} + C a$ at $(x_0, t_0)$, we get
\begin{equation}
\label{gj-S140}
\begin{aligned}
 E
   \leq \,& \sum_{i \in J} F^{ii} (\nabla_i \phi)^2
             +  C \sum_{i \in K} F^{ii} + C F^{11} \sum_{i \in K} (\nabla_i \phi)^2\\
   \leq \,& C a^2 \sum_{i \in J} F^{ii}  + C \delta^2 F^{ii} U_{ii}^2
            +  C \sum_{i \in K} F^{ii} + C (\delta^2 U_{11}^2 + a^2) F^{11}.
\end{aligned}
\end{equation}
Therefore by \eqref{gs12} and \eqref{gj-S140}, we finally obtain
\begin{equation}
\label{gs7}
\begin{aligned}
0 \geq \,& (\frac{\delta}{2} - C \delta^2 ) F^{ii} U_{ii}^2  - C a^2 \sum_{i \in J} F^{ii}
              - C (\delta^2 U_{11}^2 + a^2) F^{11} \\
              & + a \mathcal{L}(\ul u - u ) - C (1 + \sum F^{ii}).
\end{aligned}
\end{equation}
Observe that
\begin{equation}
F^{ii} U_{ii}^2 \geq F^{11} U_{11}^2 + \sum_{i \in J} F^{ii} U_{ii} \geq F^{11} U_{11}^2 + s^2 U_{11}^2 \sum_{i \in J} F^{ii}.
\end{equation}
We may firstly choose $\delta$ small sufficiently such that
$\frac{\delta}{2} - C \delta^2 > c_0 > 0$ .
Then we assume $U_{11} > R$, where $R$ is the positive constant such that \eqref{gj} holds and fix $a$ large
enough so that $a \mathcal{L}(\ul u - u ) - C (1 + \sum F^{ii}) \geq 0 $ holds,
then we would get a contradiction provided $U_{11}$ is sufficiently large from \eqref{gs7}.
Thus we get an upper bound for $U_{11}$.

\section{$C^{2}$ boundary estimates}

Throughout this section
we assume the function $\varphi \in C^{4,1} (\mathcal{P} M_T)$ is extended to
a $C^{4,1}$ function on $\overline {M_T}$, which is still denoted by $\varphi$.

Fix a point $(x_{0}, t_{0}) \in S M_T$. We shall choose a
smooth orthonormal local frame $e_1, \ldots, e_n$ around $x_0$ such that
when restricted to $\partial M$, $e_n$ is normal to $\partial M$.
Since $u - \ul{u} = 0$ on $S M_T$, we have
\begin{equation}
\label{hess-a200}
\nabla_{\alpha \beta} (u - \ul{u})
 = -  \nabla_n (u - \ul{u}) \varPi (e_{\alpha}, e_{\beta}), \;\;
\forall \; 1 \leq \alpha, \beta < n \;\;
\mbox{on  $SM_T$},
\end{equation}
where 
$\varPi$ denotes the second fundamental form of $\partial M$.
Therefore,
\[|\nabla_{\alpha \beta} u| \leq  C,  \;\; \forall \; 1 \leq \alpha, \beta < n  \;\; \mbox{on} \;\;  S M_T. \]

Let $\rho (x)$ denote the distance from $x \in M$ to $x_{0}$,
and set
\[M_T^{\delta} = \{X = (x, t) \in M \times (0,T]:
      \rho (x) < \delta, \; t \leq t_{0} + \delta\}.\]
Since $\partial M$ is smooth, we may also assume the distance function $d (x, t) \equiv d (x)$
to the boundary $SM_T$ is smooth in $M_T^{\delta}$.
\begin{lemma}
\label{lem1}
There exist some uniform positive constants $a, \delta, \varepsilon$ sufficiently small and $N$
sufficiently large such that the function
\[v = (u - \underline{u}) + ad - \frac{Nd^{2}}{2}\]
satisfies
\begin{equation}
\label{v}
\mathcal{L} v \leq - \varepsilon (1 + \sum F^{ii} + F^\tau) \mbox{ in }M_T^{\delta},
     \ v \geq 0\mbox{ on } \mathcal{P} M_T^{\delta}.
\end{equation}
\end{lemma}
\begin{proof}
We note that to ensure $v \geq 0$ in $\mathcal{P} M_{\delta}$
we may require $\delta \leq 2a/N$ after $a, N$ being fixed. It is easy to see that
\begin{equation}
\label{bs4}
\mathcal{L} v
      \leq  \mathcal{L}(u - \ul u) + C (a + Nd) \sum F^{ii}
      - N F^{ij} \nabla_i d \nabla_j d.
\end{equation}
Fix $\theta > 0$ small and $R > 0$ large enough such that \eqref{gj} holds
at every point in $\bar{M}_T^{\delta_{0}}$ for some fixed $\delta_{0} > 0$.
Let
$\lambda = \lambda [U]$ be the eigenvalues of $U$. At a fixed point in $M_{\delta}$
where $\delta < \delta_0$,
we consider two cases: (a) $|\lambda| < R$ and (b) $|\lambda| \geq R$.

In case (a), since $|u_t| \leq C$, by \eqref{f5}, there are uniform bounds
\[c_{1} I \leq \{F^{ij}\} \leq C_{1} I, \;\; c_1 \leq F^\tau \leq C_1\]
for some positive constants $c_1, C_1$,
and therefore $F^{ij} \nabla_{i} d \nabla_{j} d \geq c_{1}$ since $|\nabla d| \equiv 1$.
Since $\mathcal{L}(u - \ul u) \leq 0$, we may fix $N$ large enough
so that Lemma \ref{lem1} holds for any $a, \varepsilon \in (0,1]$, as long as $\delta$ is sufficiently small.

In case (b), since $F^{ij}\nabla_i d \nabla_j d \geq 0$, by \eqref{gj} and (\ref{bs4}) we may further require $a$ and $\delta$ small
enough so that Lemma \ref{lem1} holds.
\end{proof}

With the help of
$ \nabla_{ij} (\nabla_{k} u)
  = \nabla_{ijk} u + \Gamma_{ik}^l \nabla_{j l} u + \Gamma_{jk}^l \nabla_{ i l} u
  +  \nabla_{\nabla_{ij} e_k} u $
and \eqref{gs1}, we obtain
\begin{equation}
\label{hess-E170}
\begin{aligned}
|\mathcal{L} \nabla_k (u - \varphi)|
         \leq \,& 2 |F^{ij}  \Gamma_{ik}^l \nabla_{j l} u|  + C \Big(1 + \sum F^{ii} + F^\tau \Big) \\
              \leq \,& C \Big(1 + \sum f_i |\lambda_i| + \sum f_i + f_\tau \Big).
\end{aligned}
\end{equation}
According to \eqref{hess-E170} we have
\begin{equation}
 \mathcal{L} |\nabla_{\gamma} (u - \varphi)|^2
\geq  F^{ij} U_{i \gamma} U_{j \gamma}
         - C \Big(1 + \sum f_i |\lambda_i| + \sum f_i + f_\tau \Big).
\end{equation}
Let
\begin{equation}
\label{hess-E176}
 \varPsi
   = A_1 v + A_2 \rho^2 - A_3 \sum_{\gamma < n} |\nabla_{\gamma} (u - \varphi)|^2.
\end{equation}
For any $K > 0$, since $\nabla_l (u - \varphi) = 0$ on $\partial M$ with $1 \leq l \leq n-1$,
when $A_2 \gg A_3 \gg 1$, we have $( A_2 - K ) \rho^2 \geq A_3 \sum_{l < n} |\nabla_l (u - \varphi)|^2$
in $\overline M_T^\delta$.
Hence we can choose $A_1 \gg A_2 \gg A_3 \gg 1$ such that $\varPsi \geq K (d + \rho^2)$
in $\overline M_T^\delta$.
By Proposition \ref{par-1} and Proposition \ref{par-2} and Lemma \ref{lem1}, it follows that in $M_T^{\delta}$,
$\mathcal{L} (\varPsi \pm \nabla_{\alpha} (u - \varphi)) \leq 0$,
and $\varPsi \pm \nabla_{\alpha} (u - \varphi) \geq 0$ on $\mathcal{P} M_T^{\delta}$
when $A_1 \gg A_2 \gg A_3 \gg 1$.
 By the maximum principle we derive
$\varPsi \pm \nabla_{\alpha} (u - \varphi) \geq 0$
in $M_T^{\delta}$ and therefore
\[
|\nabla_{n \alpha} u (x_0, t_0)| \leq \nabla_n \varPsi (x_0, t_0) \leq C,
\;\; \forall \; \alpha < n.
\]

It remains to derive
$\sup_{S M_T} \nabla_{n n} u \leq C$,
since $- u_t + \triangle u + \sum \chi_{ii} \geq - C$.
For $(x, t) \in SM_T$, let $\tilde U (x, t)$ be the restriction to $T_x \partial M$ of $U (x, t)$,
viewed as a bilinear map on
the tangent space of $\partial M$ at $x$,
and $\lambda ( \tilde U (x, t) )$ be the eigenvalues with respect to the induced metric of $(M^n, g)$ on $\partial M$.
Similarly one can define $\tilde {\ul U } (x, t)$ and $\lambda (\tilde {\ul U } (x, t) )$.
On $SM_T$, we define that
\[\tilde{F} ([\tilde U, - u_t]) := \lim_{R \rightarrow \infty} f (\lambda (\tilde U), R, -u_t). \]
Due to Trudinger \cite{T}, we need only show that
the following quantity
\[m := \min_{(x, t) \in S M_T}
     \Big( \tilde F ([\tilde U, -u_t])(x, t) - \psi (x,t) \Big)\]
is positive (see \cite{G}). We can assume that $m$ is finite.
Note that $\tilde{F}$ is concave, and it is easily seen that the following holds
\[\begin{aligned}
\hbar := \,& \min_{(x, t) \in S M_T} \Big( \tilde{F}( [\tilde {\ul U}, - \ul u_t]) (x, t) - \psi (x, t) \Big) > 0,
\end{aligned}\]
and $\hbar$ may equal to infinity.
Without loss of generality we assume $m < \hbar/2$.
Suppose $m$ is achieved at a point $(x_{0}, t_{0}) \in S M_T$. Now we give some notations.
Choose a local orthonormal frame $e_1 \ldots, e_n$ around $x_0$ as before,
that is $e_n$ is normal to $\partial M$. Therefore locally we have $\tilde U = \{U_{\alpha \beta}\}$,
where $1 \leq \alpha, \beta \leq n-1$. We denote that
$\tilde{F}^{\alpha \beta}_0$ is the first order derivative of $\tilde F$ with respect to $\tilde U_{\alpha \beta}$
at $(x_0, t_0)$, and $\tilde{F}^{\tau}_0$ is the first order derivative of $\tilde F$
with respect to $- u_t$ at $(x_0, t_0)$.

By \eqref{hess-a200} we have on $ SM_T $,
\begin{equation}
\label{c-220}
 U_{\alpha {\beta}} = \ul{U}_{\alpha {\beta}}
    - \nabla_n (u - \ul{u}) \sigma_{\alpha {\beta}}
\end{equation}
where
$\sigma_{\alpha {\beta}} = \langle \nabla_{\alpha} e_{\beta}, e_n \rangle$;
note that
$\sigma_{\alpha \beta} = \varPi (e_\alpha, e_\beta)$ on
$SM_T$. 
Since on $SM_T$ we have $\ul{u}_t = u_t$,
it follows that at $(x_0, t_0)$,
\begin{equation}
\label{c-225}
\begin{aligned}
 \nabla_n (u - \ul{u}) \tF^{\alpha {\beta}}_0 \sigma_{\alpha {\beta}}
   = \,& \tF^{\alpha {\beta}}_0 (\ul{U}_{\alpha \beta} - U_{\alpha  {\beta}})
\geq \tF[\ul{U}_{\alpha \beta}, -\ul{u}_t] - \tF[U_{\alpha \beta}, -u_t] \\
  =  \,& \tF[\ul{U}_{\alpha {\beta}}, - \ul u_t] - \psi - m
 \geq \hbar - m.
\end{aligned}
\end{equation}
Setting $\eta = \tF^{\alpha {\beta}}_0 \sigma_{\alpha {\beta}}$ which is well defined in $M_T^\delta$,
by \eqref{c-225} we obtain
\begin{equation}
\label{c-230}
\eta (x_0) \geq \frac{\hbar}{2 \nabla_n (u - \ul{u}) (x_0, t_0)}
                \geq \epsilon_1 \hbar > 0
\end{equation}
for some uniform $\epsilon_1 > 0$.

Let
\[ Q \equiv \tF^{\alpha {\beta}}_0 \Big( \ul U_{\alpha {\beta}} -  U_{\alpha {\beta}} (x_0, t_0) \Big)
           + \tF^\tau_0 \Big( u_t(x_0, t_0) - \varphi_t \Big) + \psi (x_0, t_0) - \psi (x, t), \]
and we see that $Q$ is well defined in $M_T^\delta$ for small $\delta$ .
Since $\tF$ is concave, we have on $ SM_T$
\begin{equation}
\label{c-210}
\begin{aligned}
\tF^{\alpha {\beta}}_0 \Big(U_{\alpha  {\beta}} - U_{\alpha  {\beta}} (x_0, t_0)\Big)
 & - \tF_0^\tau \Big(u_t -  u_t (x_0, t_0)\Big) - F (\hat U ) + F (\hat U (x_0, t_0))\\
 & \geq \tF ([U_{\alpha  {\beta}}, -u_t]) - F (\hat U ) - m  \geq 0 .
\end{aligned}
\end{equation}
We define
\[ \varPhi = - \eta \nabla_n (u - \ul u) + Q ;\]
it follows from \eqref{c-210} and \eqref{c-220}
that $\varPhi (x_0, t_0) = 0$ and
$\varPhi \geq 0$ on $ SM_T $.
For $(x, 0) \in B M_T^\delta$ we have
\[
\varPhi (x, 0) \geq \varPhi (\hat x, 0) - C d (x) \geq - C d (x),
\]
where $C$ depends on $C^1$ norms of $\nabla^2 \ul u$, $\psi$, $\varphi_t$, and
$\hat x \in \partial M$ such that $d (x) = {dist}_{M} (x, \hat x)$ .
Besides with some calculation, we have
\begin{equation}
\label{gblq-B360}
 \begin{aligned}
\mathcal{L} \varPhi
  \leq   C \sum f_i + C \sum f_i |\lambda_i| + C f_\tau + C .
\end{aligned}
\end{equation}
Therefore, applying Proposition~\ref{par-1}, Proposition~\ref{par-2} and Lemma~\ref{lem1} again,
as well as choosing $A_1 \gg A_2 \gg A_3 \gg 1$, we derive
\begin{equation}
\label{cma-106}
  \left\{ \begin{aligned}
\mathcal {L} \,&  (\varPsi + \varPhi) \leq  0 \;\; \mbox{in $M_T^\delta$}, \\
        & \varPsi + \varPhi \geq 0 \;\; \mbox{on $\mathcal{P} M_T^\delta$}.
\end{aligned} \right.
\end{equation}
By the maximum principle,
$\varPsi +  \varPhi \geq 0$ in $M_T^\delta$. Thus
$\nabla_n \varPhi (x_0, t_0) \geq - \nabla_n \varPsi (x_0, t_0) \geq -C $.
While we also have
\begin{equation}
\label{c-255}
 - C \leq \nabla_n \varPhi (x_0, t_0)
   \leq  \Big(- \eta (x_0) \Big) \nabla_{nn} u (x_0, t_0) + C.
 \end{equation}
This with \eqref{c-230} yields that
\[\nabla_{nn} u (x_0, t_0) \leq  \frac{C}{\epsilon_1 \hbar}.\]

By now we have got \emph{a priori} upper bounds for all eigenvalues of $\{U_{ij} (x_0, t_0)\}$
and hence the eigenvalues are contained in a compact subset of $\Gamma$
by \eqref{f5}. Therefore by \eqref{f1} we obtain
$m > 0$ (for the detailed proof one can see \cite{GSS}).
This completes the proof of Theorem \ref{bd-th1}.

\section{Gradient estimates}

We now deal with the interior estimates of $|\nabla u|$ with conditions appeared in
Guan \cite{G1} and Li \cite{Li90} respectively which correspond to the two cases in Theorem \ref{bd-th2}.

Case $(\mathbf{i})$: we set
\[W = \sup_{(x, t) \in M_T}  |\nabla u|  \phi^{-\delta},\]
where $\phi = - u + \sup u + 1$ and $\delta < 1$ a positive constant. It suffices to estimate $W$. We may assume $W$ is achieved at
$(x_{0}, t_{0}) \in \overline {M_T} - \mathcal{P} M_T$.
Choose a smooth orthonormal local frame $e_1, \ldots, e_n$
about $x_0$ as before such that $\nabla_{e_i} e_j = 0$ at $x_0$.
Assume $U (x_0, t_0)$ is diagonal. Differentiating the function $\log \omega  - \delta \log \phi$
at $(x_0, t_0)$, where $\omega = |\nabla u|$, we obtain,
\begin{equation}
\label{g1}
 \frac{\nabla_i \omega}{\omega} - \frac{\delta \nabla_i \phi}{\phi} = 0, \; \mbox{ for every }i = 1, \ldots, n,
\end{equation}
\begin{equation}
\label{g0}
 \frac{\omega_t}{\omega} - \frac{\delta \phi_t}{\phi} \geq 0,
\end{equation}
and
\begin{equation}
\label{g2}
\begin{aligned}
 0 \geq \frac{\nabla_{ii}\omega}{\omega} + \frac{(\delta - \delta^2) |\nabla_i \phi|^2}{\phi^2} - \frac{\delta \nabla_{ii} \phi}{\phi}.
\end{aligned}
\end{equation}
Next, by \eqref{g1} and \eqref{hess-A70},
\begin{equation}
\label{w}
\begin{aligned}\omega \nabla_{ii} \omega
= \;& (\nabla_{lii} u + R^k_{iil} \nabla_k u)\nabla_l u + \nabla_{il} u \nabla_{il} u - \frac{\delta^2 \omega^2 (\nabla_i \phi)^2}{\phi^2}\\
\geq \;& \nabla_l U_{ii} \nabla_l u - C|\nabla u|^2
- \frac{\delta^2 \omega^2 (\nabla_i \phi)^2}{\phi^2}.
\end{aligned}
\end{equation}
Now multiply $F^{ii}$ on both side of \eqref{g2} and substitute \eqref{w} in it. It follows that
\begin{equation}
\label{w2}
\begin{aligned}
0 \geq \;& \frac{1}{\omega^2} F^{ii} \nabla_l u \nabla_l U_{ii} - C \sum F^{ii} - \frac{\delta F^{ii} \nabla_{ii} \phi}{\phi}\\
\;& + \frac{(\delta - 2\delta^2) }{\phi^2}  F^{ii} |\nabla_i \phi|^2.
\end{aligned}
\end{equation}
Now compute $F^{ii} \nabla_l u \nabla_l U_{ii}$ using \eqref{gs1}. With \eqref{g0} we have
\begin{equation}
\label{w3}
\begin{aligned}
F^{ii} \nabla_l u \nabla_l U_{ii} =  \nabla_l u \psi_{x_l} + F^{\tau} \nabla_l u \nabla_l u_t
\geq - C |\nabla u| + \frac{\delta \omega^2}{\phi} F^{\tau} \phi_t
\end{aligned}
\end{equation}
Therefore, in view of \eqref{w2} and \eqref{w3}, we obtain
\begin{equation}
\label{w4}
\begin{aligned}
0 \geq \;& - C (1 + \sum F^{ii})
+\frac{\delta}{\phi} (F^{\tau} \phi_t - F^{ii} \nabla_{ii} \phi) + \frac{(\delta - 2\delta^2) }{\phi^2} \sum F^{ii} |\nabla_i \phi|^2.
\end{aligned}
\end{equation}
Besides by \eqref{c-290}, we have
\begin{equation}
\label{w6}
F^{\tau} \phi_t - F^{ii} \nabla_{ii} \phi = \sum F^{ii}U_{ii} - F^{\tau} u_t - F^{ii} \chi_{ii} \geq - C (1 + \sum F^{ii} + F^\tau).
\end{equation}

Without loss of generality, we may assume
$|\nabla u (x_{0}, t_0)| \leq n \nabla_{1} u (x_{0}, t_0)$.
By \eqref{g1} and noting that $U(x_0, t_0)$ is diagonal, we derive at $(x_0, t_0)$
\begin{equation}
\label{w5}
U_{11} = \sum_{k \geq 2} \frac{\nabla_k u \chi_{1k}}{\nabla_1 u} - \frac{\delta |\nabla u|^2}{2 \phi} + \chi_{11}
\leq  C - \frac{\delta |\nabla u|^2}{2 \phi}.
\end{equation}
Therefore, if $|\nabla u|$ is sufficiently large (otherwise we are done), by \eqref{f6},
we obtain
\[F^{ii} (\nabla_i u)^2 \geq F^{11}|\nabla_1 u|^2 \geq \nu_0/n^2 |\nabla u|^2 (1 + \sum f_i + f_\tau). \]
 Choosing $0 < \delta < \frac{1}{2}$ and substituting the above inequality and \eqref{w6} in \eqref{w4},
 then it follows $|\nabla u(x_0, t_0)| \leq C$, and $C$ depends on $|\psi|_{C^1_x (\overline {M_T})}$,
 $|u|_{C^0 (\overline {M_T})}$, $K_0$ and other known data.

Case $(\mathbf{ii})$: Since $(M^n, g)$ has nonnegative sectional curvature, under an orthonormal local frame, i.e.
$R^k_{iil} \nabla_k u \nabla_l u \geq 0$.
In the proof of case $(\mathbf{i})$, we therefore have in place of \eqref{w},
\begin{equation}
\label{w7}
\begin{aligned}
\omega \nabla_{ii} \omega \geq  (\nabla_{lii} u + R^k_{iil} \nabla_k u)\nabla_l u
                           \geq   \nabla_l u \nabla_l U_{ii} - C |\nabla u|.
\end{aligned}
\end{equation}
Taking $\phi = u - \ul u + \sup |u - \ul u| + 1$ and $\delta < 1$, by \eqref{g0}, \eqref{g2}, \eqref{w3} and \eqref{w7}, we obtain
\begin{equation}
\label{w8}
0 \geq \frac{\delta}{\phi} \mathcal{L} (\ul u - u) + \frac{(\delta - \delta^2) }{\phi^2} \sum F^{ii} |\nabla_i \phi|^2
 - \frac{C}{|\nabla u|} (1 + \sum F^{ii})
\end{equation}
Now if $\lambda[U(x_0, t_0), - u_t (x_0, t_0)] \geq R$, by Theorem \ref{3I-th3}, we may derive a bound for $|\nabla u (x_0, t_0)|$.
If $\lambda[U(x_0, t_0), - u_t (x_0, t_0)] \leq R$, since $|u_t| \leq C$ where $C$ is independent of
$|\nabla u|$, by \eqref{f5}, there exists a positive constant $C_0$ such that
$\frac{1}{C_0} I \leq \{F^{ij}\} \leq C_0 I$,
where $I$ is the unit matrix in the set of $n \times n$ symmetric matrices $\mathbb{S}^n$.
 It follows from \eqref{w8} that
\[0 \geq \frac{(\delta - \delta^2) }{C_0 \phi^2 } |\nabla (\ul u- u)|^2 - \frac{C}{|\nabla u|} (1 + n C_0). \]
This proves $|\nabla u (x_0, t_0)| \leq C$ for $\delta < 1$.
Thus Theorem \ref{bd-th2} is completed.

\textbf{Acknowledgement:} The authors wish to thank Doctor Heming Jiao for many insightful suggestions and comments.

\end{document}